\documentclass{amsart}
\usepackage[alphabetic,lite]{amsrefs}
\usepackage{amssymb}
\theoremstyle{plain}
\newtheorem{thm}{Theorem}
\newtheorem{lem}{Lemma}
\newtheorem{prop}{Proposition}
\newtheorem{cor}{Corollary}
\theoremstyle{remark}
\newtheorem{rem}{Remark}

\def\Fl{F(A,\lambda)}  \def\Fe{F(A,\epsilon)}  
    \def\Fm{F(A,\mu)}
\def\<{\langle}  \def\>{\rangle}      
\def\la{\langle} \def\ra{\rangle}
\def\n#1{\Vert #1\Vert}
\def\N {{\mathbb N}}    \def\R {{\mathbb R}}

\def\conrems{{\it Concluding Remarks.\ }}

  \def\Lem#1{Lemma\ {#1}}
\def\z{\\[0.1cm]}  \def\zz{\\[0.2cm]}  

\def\d{\vspace*{0.1cm}}  \def\dd{\vspace*{0.2cm}} 
  \def\ddddd{\vspace*{0.5cm}}
\begin{document}
\title[\tiny A second look at a geometric proof of the spectral theorem] {A Second Look at \\
 {\it \small  A Geometric Proof of the Spectral Theorem for Unbounded Self-Adjoint Operators} }
\author{\tiny Herbert Leinfelder}
\address{Herbert Leinfelder, Nuremberg Institute of Technology, Germany}
\email{herbert.leinfelder@th-nuernberg.de}

\begin{abstract}A new geometric proof of the spectral theorem for unbounded self-adjoint operators $A$ in a Hilbert space $H$ is given based on a splitting of $A$ in positive and negative parts $A_+\geq 0$ and $A_-\leq 0$. For both operators $A_+$ and $A_-$ the spectral family can be defined immediately and then put together to become the spectral family of $A$. Of course crucial methods and results of \cite{Leinfelder1979} are used.
\end{abstract}
\maketitle
The underlying note is a second look at my article {\it A Geometric Proof of the Spectral Theorem for Unbounded Self-Adjoint Operators} that appeared 1979 in the Math.\,Ann.\,{\bf 242} \cite{Leinfelder1979}. A second look means that we concentrate in this note on semi-bounded self-adjoint operators in a Hilbert space $H$. For semi-bounded self-adjoint operators $A$, say $A\geq 0$, we can define the spectral family $(E(\lambda))_{\lambda\in\R}$ immediately by setting
$$ E(\lambda) = P_{\Fl}\ \ \ \ (\lambda\in\R)\, , $$
where $P_{\Fl}$ is the projection of $H$ onto the subspace
$$ \Fl = \{ x\,|\, x\in D(A^n),\ \n{A^n x}\leq \lambda^n \n{x}\ \ (n\in\N)\}.$$
The general situation of a totally unbounded (unbounded from above and from below) self-adjoint operator $A$ is handled by splitting the operator $A$ in two semi-bounded operators $A_+\geq 0$ and $A_-\leq 0$ such that $A = A_-\oplus A_+$. Here again a (with respect to $A$) reducing subspace $F(B+\beta,\beta)$ with $B=A(1+A^2)^{-1}$ and $B+\beta\geq 0$ is of importance.\z
Moreover we use the possibility to present certain improved and extended lemmas of \cite{Leinfelder1979}, namely \cite[Lemma 1]{Leinfelder1979} and \cite[Lemma 4]{Leinfelder1979}.
\z
Concerning definitions, notations and basic results of Hilbert spaces and in particular of the calculus of projections we refer the reader to \cite{Weidmann1980} \footnote{Notice that our scalar products are linear in the first argument.}.
\zz
We now recall those results of \cite{Leinfelder1979} that are used decisively in this note. Notice that an inspection of the proofs of these results in \cite{Leinfelder1979} shows that the basic assumption \lq$A$ symmetric\rq\ may be weakened to \lq$A$ Hermitian\rq.
\begin{lem}\label{L1}
Suppose A is a closed Hermitian operator in the Hilbert space $H$ and $\epsilon,\delta,\lambda$ are non-negative real numbers. Then
\begin{enumerate}
\item $\Fl$ is a subspace of $H$, which is left invariant by the operator $A$.
\item Every bounded linear operator $B$ satisfying $B\, A \subset A\, B$ maps $\Fl$ into itself. Similarly we have $B^*(\Fl)^\perp \subset \Fl^\perp$.
\item  $F(A+\delta,\epsilon) \subset F(A,\delta+\epsilon)$ \ and\ \ $F(A^2,\epsilon^2)= \Fe$.
\end{enumerate}
\end{lem}
\begin{proof}
For the proof of i) and ii) we refer to \cite[Lemma 1]{Leinfelder1979}.
To prove the first inclusion of iii) we know by property i) of Lemma \ref{L1} that $F(A+\delta,\epsilon)$ is invariant under $A+\delta$ and thus also under $A$ itself. Hence for $x\in F(A+\delta,\epsilon)$ we may conclude $A^nx \in F(A+\delta,\epsilon)$ for all
$n \in \N$, i.e $x \in D^\infty(A) = \bigcap_{n=1}^\infty D(A^n)$ and
\begin{eqnarray*}
\n{Ax} \leq \n{(A+\delta)x}+\delta\n{x}\leq (\epsilon+\delta)\n{x}.
\end{eqnarray*}
Suppose $ \n{A^nx} \leq (\epsilon+\delta)^n \n{x}$ is valid for given $n\in\N$ and $x\in F(A+\delta,\epsilon)$, then
\begin{eqnarray*}
\n{A^{n+1}x} &=& \n{A A^nx}\, = \, \n{(A+\delta)A^nx -\delta A^nx} \\
                      \, & \leq &\,  \epsilon\n{A^nx} +\delta\n{A^nx}\, = \,(\epsilon+\delta)\n{A^nx} \\
                      \, & \leq &\, (\epsilon+\delta)^{n+1} \n{x}.
\end{eqnarray*}
Thus by induction on $n\in\N$ we see that $x\in F(A,\delta+\epsilon)$.
\zz
To prove the second identity of iii) we first remark that $A^2$ is a closed operator .
Clearly $D^\infty(A^2) = D^\infty(A)$ and for $x\in F(A^2,\epsilon^2)$ we have
\begin{eqnarray*}
\n{A^n x}^2 \, & = & \, \la A^nx,A^nx \ra \,\, = \,\, \la x,A^{2n}x \ra \\
         \, & \leq & \, \n{x}\,\, \n{(A^2)^n x}\, \leq \, \n{x}\, (\epsilon^2)^n \n{x},
\end{eqnarray*}
which gives $\n{A^nx} \leq \epsilon^n \n{x}$ and thus $x\in F(A,\epsilon)$. This gives $F(A^2,\epsilon^2) \subset \Fe$. The inverse inclusion $\Fe \subset F(A^2,\epsilon^2)$ is obvious. \qed
\end{proof}
\begin{lem}\label{L2}
Suppose $A$ is a Hermitian operator in a finite dimensional Hilbert space $H$. Then we have for all $x\in\Fl^\perp, x\neq 0$, and all $\lambda\geq 0$:
\begin{enumerate}
  \item $\n{Ax} > \lambda\, \n{x}$,
  \item $\la Ax,x\ra > \lambda\, \la x,x\ra$, \ provided $A\geq 0$
\end{enumerate}
\end{lem}
\begin{proof}
For a proof see \cite[Lemma 2]{Leinfelder1979}. \qed
\end{proof}
\begin{prop}\label{P}
Suppose $A$ is a closed Hermitian operator in a Hilbert space $H$ and $\lambda,\mu$ are non-negative real numbers. Then for every $x\in \Fm\cap\Fl^\perp$ there is a sequence $(A_n)$ of Hermitian operators $A_n : H_n \rightarrow H_n$ defined in finite dimensional subspaces $H_n\subset H$ and a sequence $(x_n)$ of Elements $x_n$ belonging to $H_n$ in such a way that
$$
x_n\in F(A_n,\lambda)^\perp \ \mbox{and}\ \lim_{n\to\infty}\n{x_n-x} = 0 = \lim_{n\to\infty}\n{Ax_n - Ax}.
$$
If $A\geq 0$ then $A_n$ can be chosen non-negative.
\end{prop}
\begin{proof} For a proof see \cite[Lemma 3]{Leinfelder1979}. \qed
\end{proof}
\begin{lem}\label{L3}
Let $A$ be a closed Hermitian operator in $H$, and $\lambda,\mu $ nonnegative real numbers. Then for all
$x\in\Fm\cap\Fl^\perp$
\begin{enumerate}
  \item $\lambda\, \n{x} \leq \n{Ax} \leq \mu\,\n{x}$,
  \item $\lambda\, \la x,x\ra \leq \la Ax,x\ra \leq \mu\, \la x,x\ra$, provided $A\geq 0$.
\end{enumerate}
\end{lem}
\begin{rem}
Note that Lemma \ref{L3} i) can be written in the equivalent form
\begin{equation*}
    \lambda^2\, \la x,x\ra \leq \la A^2x,x\ra \leq \mu^2\, \la x,x\ra.
\end{equation*}
\end{rem}\d
\begin{proof}
For $x\in\Fm$ the inequalities $\n{Ax} \leq \mu\,\n{x}$ and $\la Ax,x\ra \leq \mu\, \la x,x\ra$ follow directly from the definition of $\Fm$ whereas for $x\in\Fm\cap\Fl^\perp$ the inequalities $\lambda\, \n{x} \leq \n{Ax}$ and $\lambda\, \la x,x\ra \leq \la Ax,x\ra$ can be derived from  Lemma \ref{L2} using the limiting process described in Proposition \ref{P}. \qed
\end{proof}\d
\begin{lem}\label{L4}
Let $A$ be a closed Hermitian operator in a Hilbert space $H$.
Then the following statements are equivalent:
\begin{enumerate}
\item $\bigcup \limits_{\epsilon>0} F(A,\epsilon)$ is dense in $H$.
\item $A$ is self-adjoint.
\end{enumerate}
\end{lem}
\begin{proof}
The proof of $i) \Rightarrow ii)$  can be found in \cite[Lemma 4]{Leinfelder1979} and remains unchanged. To see $ii)\Rightarrow i)$ we present a new shortened and simplified proof.
\z
Suppose $A$ is self-adjoint and assume in addition that $A$ is bounded from below by 1, i.e. $A \geq 1$. Then $A^{-1}$ exists, is bounded and the following relation holds true:
\begin{eqnarray}\label{INV}
F(A^{-1},\epsilon^{-1})^\perp \subset F(A,\epsilon)\ \ \ (\epsilon>0)
\end{eqnarray}
Assume for the moment that (\ref{INV}) is already proven then
\begin{eqnarray*}
   \Biggl(\bigcup \limits_{\epsilon>0} F(A,\epsilon)\Biggr)^\perp \, = \, \bigcap\limits_{\epsilon>0}F(A,\epsilon)^\perp \, \subset \, \bigcap \limits_{\epsilon>0} F(A^{-1},\epsilon^{-1})\,=\, N(A^{-1})\, = \, \{0\}
\end{eqnarray*}
and $\bigcup \limits_{\epsilon>0} F(A,\epsilon)$ is dense in $H$.

For general $A$ we consider the self-adjoint operator $S = A^2+1 \geq 1$ being self-adjoint in view of \cite[Theorem 3.24]{Kato1976}. We use the inclusions iii) in Lemma \ref{L1} to get
\begin{eqnarray*}
    F(S,\epsilon) \, \subset \, F(A^2,\epsilon+1)\, \subset \,  F(A,\sqrt{\epsilon+1}) \ \ \ (\epsilon > 0)
\end{eqnarray*}
and the density of $\bigcup \limits_{\epsilon>0} F(A,\epsilon)$ in $H$ is proven. \z
To end the proof we have to justify inclusion (\ref{INV}). In order to show (\ref{INV}) we consider
$H_\epsilon = F(A^{-1},\epsilon^{-1})^\perp$ and $B_\epsilon = A^{-1}|H_\epsilon$. Applying assertion $ii)$ of \Lem{1} we have $B_\epsilon(H_\epsilon) \subset H_\epsilon$
and using assertion $i)$ of Lemma 3 with $A$ replaced by $A^{-1}$, $\lambda = \epsilon^{-1}$ and $\mu = \n{A^{-1}}$ we conclude that $B_\epsilon\, :H_\epsilon \longrightarrow H_\epsilon $ is bijective and $B_\epsilon^{-1} = A|H_\epsilon$ with\  $\n{A|H_\epsilon} \leq \epsilon$.\ Hence $\n{(A|H_\epsilon)^n} \leq \epsilon^n$ for $n\in\N$ and consequently
$\n{A^nx} \leq \epsilon^n \, \n{x}$ for $x\in H_\epsilon$. It follows $x \in F(A,\epsilon)$ which implies inclusion (\ref{INV}). \qed
\end{proof}\dd
\begin{rem}\label{rem}
A closed Hermitian operator $A$ is self-adjoint if and only if $A^2$ is self-adjoint.
\end{rem}\d
\begin{proof}
This follows from Lemma \ref{L4} and the identity $F(A^2,\epsilon^2) = \Fe$ being valid in view of Lemma \ref{L1} for all $\epsilon \geq 0$. Notice that for a closed Hermitian operator $A$ the operator $A^2$ is always closed.  \qed
\end{proof}
\begin{lem}\label{L4.5}
Suppose $A$ is a closed Hermitian operator in a Hilbert space H and $(P_n)_{n\in\N}$ an increasing sequence of projections such that $R(P_n)\subset D(A), AP_n = P_nAP_n$ and $P_n\to I$ strongly. Then the following assertions hold:
\begin{enumerate}
  \item $D(A) = \{x\in H \ |\ (AP_nx)\ converges \} = \{x\in H\ |\  (\n{AP_nx})\ converges\}$
  \item $ Ax = \lim_{n\to\infty} AP_nx\ for\ all\ x\in D(A) $
\end{enumerate}
\end{lem}
\begin{proof}
For a proof see \cite[Lemma 5]{Leinfelder1979}. \qed
\end{proof}
\begin{lem} \label{L5}
Suppose $A$ is a self-adjoint operator in a Hilbert space $H$. Then there exist subspaces $H_\pm \subset H$,  self-adjoint operators $A_\pm = A_{D(A)\cap H_\pm}$ in $H_\pm$ such that
\begin{equation}\label{deco}
    H = H_- \oplus H_+,\ A = A_- \oplus A_+ \mbox{\ and\ } A_- \leq\ 0\ \leq A_+ \, .
\end{equation}
\end{lem}
\vspace*{0.2cm}
\begin{proof}
We put $B = A(1+A^2)^{-1}$ and $E = P_{F(B+\beta,\beta)}$ with $\beta \geq 0$ such that $B+\beta \geq 1$, where $E$ is the projection of $H$ onto the subspace $F(B+\beta,\beta)$. Let us note, that $B$ is a bounded self-adjoint operator. \newline
From Lemma \ref{L1} ii) with A replaced by $B+\beta$ and $B$ replaced by $E$ we conclude $E B =B E$. Hence
\begin{equation}\label{eqL5.1}
  E A (A^2+1)^{-1} = E B = B E = A (A^2+1)^{-1} E = A E (A^2+1)^{-1}
\end{equation}
where we used $(A^2+1)^{-1} E = E (A^2+1)^{-1}$, being valid by Lemma \ref{L1} ii) since $B$ and $(A^2+1)^{-1}$ commute. If we drop the middle terms in \eqref{eqL5.1} and apply both sides of the resulting equation to $y = (A^2+1)x$ with $x\in D(A^2)$ we get
\begin{equation}\label{eqL5.2}
  E A x = A E x \ \ \ (x\in D(A^2))
\end{equation}
Let us extend \eqref{eqL5.2} to elements $x\in D(A)$. Since $D(A^2)$ is a core of $D(A)$
(see \cite[Theorem 3.24]{Kato1976}) there is for each $x\in D(A$) a sequence $(x_n)\subset D(A^2)$ such that $x_n \to x, A x_n \to Ax$.
Using \eqref{eqL5.2} we conclude $E x_n \to E x \mbox{ as well as \emph{}} A E x_n = E A x_n \to E A x$.
Since $A$ is a closed operator $E x \in D(A)$ and
\begin{equation}\label{eqL5.3}
  A E x = E A x  \ \ \ (x\in D(A)),
\end{equation}
which means $E A \subset A E$.
\zz
We put $H_- = R(E)$, $H_+ = R(I-E)$ which gives $H = H_-\oplus H_+$ and we remind the reader that $E = P_{H_-}$ is the projection of $H$ onto the subspace $H_-$. Because of $E A \subset A E$ the subspaces $H_-$ and $H_+$ are reducing subspaces for the self-adjoint operator $A$ and the operators
$$A_- = A|_{D(A)\cap H_-} \mbox{\ \ and \ \ } A_+ = A|_{D(A)\cap H_+}$$
are self-adjoint operators in the Hilbert spaces $H_-$ and $H_+$ with $A = A_- \oplus A_+$. See
\cite[Theorem 7.28]{Weidmann1980} for details of this facts.
\zz
Now it remains to prove $A_-\leq 0$ and $A_+\geq 0$. For all $x\in D(A)$ we have
\begin{equation*}
    Ax = (A^2+1)(A^2+1)^{-1}Ax = (A^2+1)A(A^2+1)^{-1}x = (A^2+1)Bx
\end{equation*}
which gives\ \ $\la Ax,x\ra = \la Bx,x\ra + \la A^2Bx,x\ra $\ \ and thus
\begin{equation}
    \la Ax,x\ra = \la Bx,x\ra + \la B Ax,Ax\ra  \label{qf}
\end{equation}
where we have used $B A \subset A B$. Now for $x\in H_- = F(B+\beta,\beta)$ we have because of Lemma \ref{L3} ii) with $A$ replaced by $B+\beta$ and $\lambda= 0, \mu = \beta$
\begin{equation*}
    \la(B+\beta)x,x\ra  \leq \beta \la x,x\ra
\end{equation*}
or
\begin{equation}\label{pos}
  \la Bx,x\ra  \leq 0.
\end{equation}
Hence for $x\in D(A_-)$  we have $Ax\in H_-$ and thus in view of \eqref{pos} the inequality
\begin{equation}\label{pos1}
  \la BAx,Ax\ra \leq 0
\end{equation}
holds true. Together with equation \eqref{qf} we obtain
$\la Ax,x\ra \leq 0$, i.e. $A_-\leq 0$.
\zz
In analogue way we conclude $\la Ax,x\ra \geq 0$ for all $x\in D(A_+)$, i.e. $A_+\geq 0$.
Indeed for $x\in H_+ = F(B+\beta,\beta)^\perp$ we have because of Lemma \ref{L3} ii) with $A$ replaced by $B+\beta$ and $\lambda=  \beta, \mu = \n{B+\beta}$ the inequality
\begin{equation*}
  \beta \la x,x\ra  \leq \la(B+\beta)x,x\ra
\end{equation*}
or
\begin{equation}\label{pos2}
 0 \leq \la Bx,x\ra .
\end{equation}
For $x\in D(A_+)$ its clear that $A_+x = Ax \in H_+$ and thus
\begin{equation}\label{pos3}
   0 \leq \la BAx,Ax\ra
\end{equation}
which together with \eqref{pos2} and \eqref{qf} gives $0 \leq \la Ax,x\ra $, i.e. $0 \leq A_+$.
\qed
\end{proof}
\begin{cor}\label{cor}
Suppose the self-adjoint operators $A_-$ and $A_+$ in Lemma \ref{L5} admit spectral representations
$ A_\pm = \int_\R \lambda \, d E_\pm(\lambda) $, then with $E(\lambda) = E_-(\lambda) \oplus E_+(\lambda)$ the operator $ A = A_- \oplus A_+$ admits a spectral representation
$$A = \int_\R \lambda \, d E(\lambda).$$
\end{cor}
\begin{proof}
Let us write
\begin{gather}\label{not}
    x = x_- + x_+ ,\ \  x\in D(A),\ x_\pm \in D(A_\pm),\ \  Ax = A_-x_- + A_+x_+ \\
    A_\pm = \int_\R \lambda\, dE_\pm(\lambda) \ \ \mbox{ with spectral families }\ \   (E_\pm(\lambda)_{\lambda\in\R}.
\end{gather}
We have \ \ \   $ x\in D(A) \Leftrightarrow x_\pm\in D(A_\pm) \Leftrightarrow  \int_\R \lambda^2\,dE_\pm(\lambda) < \infty $. \z
If we define \ $E(\lambda) = E_-(\lambda)\oplus E_+(\lambda)$ \ then in view of
\begin{gather*}
    \la E(\lambda)x,y \ra = \la E_-(\lambda)x_-,y_-\ra + \la E_+(\lambda)x_+,y_+\ra \\
    \la Ax,y \ra = \la A_-x_-,y_- \ra + \la A_+x_+,y_+ \ra
\end{gather*}
we conclude
\begin{gather*}
    x\in D(A) \Leftrightarrow \int_\R \lambda^2\,dE_\pm(\lambda) < \infty \Leftrightarrow \int_\R \lambda^2\,d\,E(\lambda) < \infty . \\
    \la Ax,y \ra = \la A_-x_-,y_- \ra + \la A_+x_+,y_+ \ra  = \int_\R \lambda\, d\la E(\lambda)x,y\ra.
\end{gather*}
\qed
\end{proof}
\begin{thm} \label{T}
Every self-adjoint operator $A$ in a Hilbert space $H$ admits one and only one spectral family
$(E(\lambda)_{\lambda\in\R}$ such that
$$  A = \int_\R \lambda \, d E(\lambda) $$
\end{thm}
\begin{proof}
{\it Uniqueness.}\ The proof of the uniqueness of the representing spectral family is given in
\cite[Theorem 1]{Leinfelder1979}. \zz
{\it Existence.}\ We first prove the spectral theorem for positive self-adjoint operators $A>0$. We define a family of projections by setting
\begin{equation}\label{fam}
    E(\lambda) = P_{\Fl} \ ( \lambda\in\R).
\end{equation}
Notice that $E(\lambda) = 0$ if $\lambda \leq 0$. It is not difficult to check that $(E(\lambda))_{\lambda\in\R}$ is actually a spectral family. The only nontrivial point is the property $\lim_{\lambda\to\infty}E(\lambda) = I$. But this follows from Lemma \ref{L4}. So it remains to show the validity of the formula
\begin{equation}\label{spec}
     A = \int_\R \lambda \, d E(\lambda).
\end{equation}
We take $n\in\N$, $x\in F(A,n)$ and fix these elements. Also for fixed $k\in\N$ we define $\lambda_i = \frac{i}{k}$ for all $i \in \{0,1,2,\dots,n k\}$. With
\begin{equation}\label{def1}
    x_i = E(\lambda_i)x - E(\lambda_{i-1})x \in F(A,\lambda_i)\cap F(A,\lambda_{i-1})^\perp \mbox{ for all } 1\leq i \leq nk.
\end{equation}
we have
\begin{equation*}
    x = \sum_{i=1}^{nk} x_i  \mbox{\ \ and \ \ } \n{x}^2 = \sum_{i=1}^{nk} \n{x_i}^2,
\end{equation*}
since $\la x_i,x_j\ra = 0$ if $i\neq j$. \z
We use Lemma \ref{L3} to obtain the following inequalities.
\begin{gather}
    \label{est1}
    \lambda_{i-1}\la x_i,x_i\ra \leq \la Ax_i,x_i \ra \leq \lambda_i\la x_i,x_i\ra \ \ \ (1\leq i\leq n k) \\
    \label{est2}
    \lambda^2_{i-1}\la x_i,x_i\ra \leq \la A^2x_i,x_i \ra \leq \lambda^2_i\la x_i,x_i\ra \ \ \ (1\leq i\leq n k)
\end{gather}
hence
\begin{gather}
    \label{est3}
    |\la(A-\lambda_i)x_i,x_i\ra| \leq (\lambda_i -\lambda_{i-1})\ \n{x_i}^2 \leq \frac{1}{k}\ \n{x_i}^2 \ \ \ (1\leq i\leq n k) \\
    \label{est4}
    |\la(A^2-\lambda^2_i)x_i,x_i\ra| \leq (\lambda_i^2 -\lambda_{i-1}^2)\ \n{x_i}^2 \leq   \frac{2n}{k}\ \n{x_i}^2 \ (1\leq i\leq n k)
\end{gather}
The identity \ $\la A x,x \ra  =  \sum_{i=1}^{n k} \la Ax_i,x_i \ra$\ and equation \eqref{est3} gives
\begin{gather*}
    \left|\la Ax,x \ra - \sum_{i=1}^{n k} \lambda_i\la x_i,x_i \ra\right| = \left|\sum_{i=1}^{n k} \la (A-\lambda_i)x_i,x_i \ra\right| \leq \\
    \sum_{i=1}^{n k} \left|\la (A-\lambda_i)x_i,x_i \ra\right| \leq \frac{1}{k} \sum_{i=1}^{nk}\ \n{x_i}^2 =
     \frac{1}{k}\ \n{x}^2
\end{gather*}
A similar estimate (using \eqref{est4}) holds for $A^2$ so that we have the following set of estimates
\begin{gather}\label{est5}
  \left|\la Ax,x \ra - \sum_{i=1}^{n k}\lambda_i \la x_i,x_i \ra\right| \leq  \frac{1}{k}\ \n{x}^2 \\
  \left|\la A^2x,x \ra - \sum_{i=1}^{n k} \lambda^2_i\la x_i,x_i \ra\right| \leq  \frac{2n}{k}\ \n{x}^2
\end{gather}
For fixed $n\in\N$ we let tend $k\to\infty$ and obtain for $x\in F(A,n)$
\begin{gather}\label{STna}
    \la Ax,x\ra = \int_0^n \lambda\ d\la E(\lambda)x,x\ra =\int_{[0,n]}\lambda\ d\la E(\lambda)x,x\ra  \\
    \n{Ax}^2 = \la A^2x,x \ra = \int_0^n\lambda^2\, d\la E(\lambda)x,x\ra =\int_{[0,n]}\lambda^2\, d\la E(\lambda)x,x\ra  \label{STnb}
\end{gather}
Notice that the first integrals in \eqref{STna}, \eqref{STnb} are Riemann-Stieltjes integrals and the second ones are Lebesgue-Stieltjes integrals. \z
Now we take an arbitrary $x\in D(A)$, put $P_n = E(n)$ then $P_n x\in F(A,n)$ and thus using  \eqref{STna} and
\eqref{STnb} we get
\begin{gather}\label{STna1}
    \la AP_n x,P_n x\ra = \int_{[0,n]}\lambda\ d\la E(\lambda)P_n x,P_n x\ra  = \int_{[0,n]}\lambda\ d\la E(\lambda)x,x\ra  \\
    \n{AP_n x}^2 = \la A^2x,x \ra = \int_{[0,n]}\lambda^2\, d\la E(\lambda)x,x\ra  \label{STnb1}
\end{gather}
since $\ P_n E(\lambda)P_n = E(\lambda)\ $\ for\ $0 \leq \lambda \leq n$. In a last step we apply Lemma \ref{L4.5} together with \eqref{STna1} and \eqref{STnb1} to get
\begin{equation*}
   A = \int_\R \lambda\ dE(\lambda),
\end{equation*}
 i.e. for exactly $x\in D(A)$ we have
\begin{gather}\label{STsa}
    \n{Ax}^2 = \int_{\R}\lambda^2\, d\la E(\lambda)x,x\ra\ <\ \infty , \\
    \la Ax,x\ra = \int_{\R}\lambda\ d\la E(\lambda)x,x\ra  \label{STsb}.
\end{gather}
Notice that actually $ E(\lambda) = 0 $ for all $\lambda \leq 0$.\zz
Now let us extend the validity of a spectral representation to arbitrary self-adjoint operators.
First we remind the reader of the simple fact that if a self-adjoint operator $A$ has a spectral representation
$\ A = \int_{\R} \lambda\ d\,E(\lambda)\ $ then the families
\begin{equation*}
   F(\lambda) := E(\lambda - c)) \mbox{\ \  as well as\ \ } G(\lambda) := I - E(-\lambda)\ \ \ \ \ (\lambda\in\R)
\end{equation*}
are spectral resolutions for $ A+c$ and $ -A $ respectively. Notice that $G(\lambda)$ is actually left-continuous. So to be formally correct one has in fact to chose the right-continuous spectral family $(G(\lambda+))_{\lambda\in\R}$. \newline
We know now that every semi-bounded self-adjoint operator admits a spectral representation. In a last step we apply Lemma \ref{L5} and Corollary \ref{cor} to guarantee the existence of a spectral representation for all self-adjoint operators. \qed
\end{proof}
\ddddd
\conrems \ 1. Let us mention that \eqref{STsb} can be extended (by using the polarization identity) to
\begin{equation}\label{STsc}
  \la Ax,y\ra = \int_{\R}\lambda\ d\la E(\lambda)x,y\ra\ \ \ \ \ (x,y\in D(A)) .
\end{equation}
Even elements $y\in H$ are allowed in \eqref{STsc} (by a specific interpretation of the integral), but we will not prove this here.
\zz
2. As a consequence of the proof of Lemma \ref{L4} the spectral theorem is now
proven directly in real as well as in complex Hilbert spaces. (The existence of a resolvent $R(A,\lambda)$ with {\bf non-real} $\lambda$ is no more needed!)
%


\begin{bibdiv}
\begin{biblist}

\bib{Kato1976}{book}{
      author={Kato, T.},
       title={Perturbation theory for linear operators},
      series={Grundlehren der mathematischen Wissenschaften},
   publisher={Berlin, Heidelberg, New York: Springer},
        date={1976},
      volume={{\bf 132}},
}

\bib{Leinfelder1979}{article}{
      author={Leinfelder, H.},
       title={A {G}eometric {P}roof of the {S}pectral {T}heorem for
  {S}elf-{A}djoint {U}nbounded {O}perators},
        date={1979},
     journal={Math. Ann.},
      volume={{\bf 242}},
       pages={85\ndash 96},
}

\bib{Weidmann1980}{book}{
      author={Weidmann, J.},
       title={{L}inear {O}perators in {H}ilbert {S}paces},
      series={Graduate Texts in Mathematics},
   publisher={New York, Heidelberg, Berlin: Springer},
        date={1980},
      volume={{\bf 68}},
}
\end{biblist}
\end{bibdiv}
\end{document}